\documentclass[a4paper,draft,notitlepage,12pt]{article}
\usepackage{amssymb}
\usepackage{amsfonts}
\usepackage{amsmath}
\usepackage{indentfirst}
\usepackage{theorem}

{\theorembodyfont{\upshape}

\newtheorem{example}{Example}[section]
}
\newtheorem{theorem}{Theorem}[section]

\newtheorem{definition}{Definition}[section]
\newtheorem{lemma}{Lemma}[section]

\newenvironment{proof}[1][Proof]{\noindent\textbf{#1.} }{\ \rule{0.5em}{0.5em}}

\makeatletter \@addtoreset{equation}{section}
\renewcommand{\theequation}{\thesection.\@arabic\c@equation}
\@addtoreset{figure}{section}
\renewcommand{\thefigure}{\thesection.\@arabic\c@figure}
\@addtoreset{table}{section}
\renewcommand{\thetable}{\thesection.\@arabic\c@table}
\makeatother
\input{tcilatex}

\begin{document}

\title{Convergence in the Boundary Layer for Nonhomogeneous Linear
Singularly Perturbed Systems\thanks{%
This work was supported by the Harbin Institute of Technology Science
Foundation under grant HITC200712.}}
\author{Zhibin Yan \\
{\small Center for Control and Guidance, Harbin Institute of Technology,
Harbin, 150001, China}\\
{\small \ E-mail: zbyan@hit.edu.cn}}
\date{}
\maketitle

\begin{abstract}
Convergence of the solutions of nonhomogeneous linear singularly perturbed
systems to that of the corresponding reduced singular system on the
half-line [0, $\infty $) is considered. To include the situation on a
neighborhood of initial instant, a boundary layer,\ a distributional
approach to convergence is adopted. An explicit analytical expression for
the limit as a distribution is proved.


\end{abstract}

\section{Introduction}

A rational motivation to study singular linear system, 
\begin{equation}
E\dot{x}=Ax+Bu  \label{SimplySys}
\end{equation}%
with \emph{singular} matrix $E,$ is that it is an evident simplification of
the singularly perturbed systems%
\begin{equation}
E(\epsilon )\dot{x}=Ax+Bu  \label{EpsilonSys}
\end{equation}%
for a \textquotedblleft small\textquotedblright\ parameter $\epsilon $ (may
be of vector form), where $E(\epsilon )\ $is \emph{nonsingular} and tends to 
$E\ $as $\epsilon \rightarrow 0$. The system (\ref{EpsilonSys}) arises
naturally from, for example, coupling subsystems with \textquotedblleft
slowly\textquotedblright\ and \textquotedblleft fastly\textquotedblright\
varying states respectively, optimal linear-quadratic regulator with cheap
control, etc. For detail, see \cite{Verghese1981}--\cite{Cobb2006}. For a
specific system analysis or synthesis problem, the effectiveness of the
above simplification relys on \textquotedblleft approximate
extent\textquotedblright\ between the solution to the problem for (\ref%
{SimplySys}) and that for (\ref{EpsilonSys}). Partially for characterizing
\textquotedblleft approximate extent\textquotedblright\ in the singular
perturbation analysis, some interesting topologies are introduced. See, for
example, \cite{Cobb2006}, \cite{Hinrichsen1997} and the references therein.

In this paper, we are interested in the following singularly perturbed
initial value problem%
\begin{equation}
\begin{array}{rll}
N(\epsilon )\dot{x}(t) & = & x(t)+f(t),\ t\geq 0 \\ 
x(0) & = & x_{0},%
\end{array}
\label{eqPerturbation}
\end{equation}%
and the corresponding reduced one%
\begin{equation}
\begin{array}{rll}
N\dot{x}(t) & = & x(t)+f(t),\ t\geq 0 \\ 
x(0) & = & x_{0},%
\end{array}
\label{eqdiffer}
\end{equation}%
Here $N(\epsilon )\in {\mathbb{R}}^{n\times n}$ is nonsingular for $\epsilon
\neq 0$ and tends to $N,$ a nilpotent matrix, as $\epsilon \rightarrow 0.$
The index of nilpotency of $N$ is denoted by $q,$ i.e.,%
\begin{equation}
q=\min \{k:k\geq 1,N^{k}=0\}.
\end{equation}%
The nonhomogeneous term $f$ is a $q-1$ times continuously differentiable
function mapping ${\mathbb{R}}_{+}=[0,+\infty )$ to ${\mathbb{R}}^{n}$.
Under a regularity assumption, the singular system (\ref{SimplySys}) can be
transformed into two subsystems through Weierstrass decomposition \cite%
{Gantmacher}. One has the form of the normal linear system which has trivial
relationship to the corresponding perturbed ones, and another is of the form
(\ref{eqdiffer}). For more detail of background, see \cite{Dai1989}. For
general initial conditions (\textquotedblleft inconsistent initial
conditions\textquotedblright ), the problem (\ref{eqdiffer}) has no solution
in the sense of classical differentiable function, and the corresponding
physical system exhibits impulsive behavior \cite{Verghese1981}. Thus some
generalized solutions are adopted for the problem (\ref{eqdiffer}). Recently 
\cite{Muller2005}--\cite{Yan2007}, an explicit distributional solution of (%
\ref{eqdiffer}),%
\begin{equation}
x(t)=-\tsum\nolimits_{i=0}^{q-1}N^{i}f^{(i)}(t)-\tsum\nolimits_{k=1}^{q-1}%
\delta ^{(k-1)}(t)N^{k}\left\{
x_{0}+\tsum\nolimits_{i=0}^{q-1}N^{i}f^{(i)}(0)\right\} ,\text{ }t\geq 0,
\label{Distributionsolution}
\end{equation}%
is obtained by Laplace transform. So in what sense and whether the solution
of (\ref{eqPerturbation}) given by%
\begin{equation}
x_{\epsilon }(t)=\exp \{(N(\epsilon
))^{-1}t\}x_{0}+\tint\nolimits_{0}^{t}\exp \{(N(\epsilon ))^{-1}(t-\tau
)\}(N(\epsilon ))^{-1}f(\tau )d\tau ,\text{ }t\geq 0,
\label{Perturbedsolution}
\end{equation}%
a classical function mapping ${\mathbb{R}}_{+}$ to ${\mathbb{R}}^{n},$ can
be approximated by the distribution (\ref{Distributionsolution}) becomes
interesting.

Works \cite{Francis1982} and \cite{Cobb1982} have the same concern, but they
only considered natural response (i.e., the solution for $f=0$). The forced
response (i.e., the solution for $x_{0}=0$) of (\ref{eqdiffer}) also
contains impulse term at initial instant according to (\ref%
{Distributionsolution}). So the convergence in a neighborhood of $t=0$, a
\textquotedblleft boundary layer\textquotedblright\ (region of nonuniform
convergence, see\ \cite{Francis1982}, \cite{O'Malley1979}), for the forced
response also appeal to a distributional approach. This motivates a
generalization to the results in \cite{Cobb1982} to include the
nonhomogeneous case. For other related works, see \cite{Kokotovic1976}, \cite%
{Kokotovic1980}, \cite{O'Malley1979}, \cite{Campbell1980} and the references
therein.

\section{\label{Section2}Notations and Definitions}

We review some notations and definitions in distribution theory \cite%
{Gelfand1965}. Let $\mathcal{C}_{C}^{\infty }(\mathbb{R},\mathbb{R}^{n})$ be
the space of infinitely differentiable functions from $\mathbb{R}$ to $%
\mathbb{R}^{n}$ with compact support. There is a topology on it \cite%
{Gelfand1965}, and then the distribution space is defined as the dual space $%
\mathcal{C}_{C}^{\infty }(\mathbb{R},\mathbb{R}^{n})^{\prime }$. So a
distribution $w\in \mathcal{C}_{C}^{\infty }(\mathbb{R},\mathbb{R}%
^{n})^{\prime }$ is a linear continuous functional on $\mathcal{C}%
_{C}^{\infty }(\mathbb{R},\mathbb{R}^{n}).$ The value, a real number, of $w$
on $\lambda \in \mathcal{C}_{C}^{\infty }(\mathbb{R},\mathbb{R}^{n})$ will
be denoted by $\left\langle w,\ \lambda \right\rangle .$ The Dirac delta
distribution $\delta \in \mathcal{C}_{C}^{\infty }(\mathbb{R},\mathbb{R}%
)^{\prime }$ is defined by $\left\langle \delta ,\ \lambda \right\rangle
=\lambda (0)$ for $\forall \lambda \in \mathcal{C}_{C}^{\infty }(\mathbb{R},%
\mathbb{R}).$ For any distribution $w\in \mathcal{C}_{C}^{\infty }(\mathbb{R}%
,\mathbb{R})^{\prime },$\ its $k$-th order distributional derivative $%
\mathcal{D}_{\mathrm{d}}^{(k)}w\in \mathcal{C}_{C}^{\infty }(\mathbb{R},%
\mathbb{R})^{\prime }$ is defined by%
\begin{equation}
\left\langle \mathcal{D}_{\mathrm{d}}^{(k)}w,\ \lambda \right\rangle
=(-1)^{k}\left\langle w,\ \lambda ^{(k)}\right\rangle  \label{distribudaoshu}
\end{equation}%
for $\forall \lambda \in \mathcal{C}_{C}^{\infty }(\mathbb{R},\mathbb{R}),$
where $\lambda ^{(k)}$ denotes the $k$-th order usual derivative. Let $%
\mathcal{L}_{\mathrm{loc}}(\mathbb{R},\mathbb{R}^{n})$ denote the set of all
locally Lebesque integrable functions from $\mathbb{R}$ to $\mathbb{R}^{n}.$
The embedding map $\mathcal{E}:\mathcal{L}_{\mathrm{loc}}(\mathbb{R},\mathbb{%
R}^{n})$ $\rightarrow \mathcal{C}_{C}^{\infty }(\mathbb{R},\mathbb{R}%
^{n})^{\prime }$ is defined by $\left\langle \mathcal{E}z,\ \lambda
\right\rangle =\tint\nolimits_{-\infty }^{+\infty }z(t)^{\mathrm{T}}\lambda
(t)dt$ for $\forall z\in \mathcal{L}_{\mathrm{loc}}(\mathbb{R},\mathbb{R}%
^{n})$ and $\forall \lambda \in \mathcal{C}_{C}^{\infty }(\mathbb{R},\mathbb{%
R}^{n}).$ Here $z(t)^{\mathrm{T}}$ represents the transpose of $z(t),$ and
the integral is in the sense of Lebesque. We do not distinguish $z$ and $%
\mathcal{E}z$ in following.\ Lastly, let $\mathcal{C}^{k}({\mathbb{R}}_{+},{%
\mathbb{R}}^{n})$ denote the set of all $k$-times continuously
differentiable functions from ${\mathbb{R}}_{+}$ to ${\mathbb{R}}^{n}$,
which can be seen as a subset of $\mathcal{L}_{\mathrm{loc}}(\mathbb{R},%
\mathbb{R}^{n})$ naturally.

Now we cite the definition of convergence of distribution sequence \cite%
{Gelfand1965}.

\begin{definition}
\label{DefOne}Given sequence $\{z_{i}\}_{i=1}^{\infty }\subset \mathcal{C}%
_{C}^{\infty }(\mathbb{R},\mathbb{R}^{n})^{\prime }$ and $z\in \mathcal{C}%
_{C}^{\infty }(\mathbb{R},\mathbb{R}^{n})^{\prime },$\ then $%
\{z_{i}\}_{i=1}^{\infty }$ is said to converge to $z\ $in $\mathcal{C}%
_{C}^{\infty }(\mathbb{R},\mathbb{R}^{n})^{\prime },$ denoted by $%
\lim\nolimits_{i\rightarrow \infty }z_{i}=z,$ if for every $\lambda \in 
\mathcal{C}_{C}^{\infty }(\mathbb{R},\mathbb{R}^{n}),$%
\begin{equation*}
\lim\nolimits_{i\rightarrow \infty }\left\langle z_{i},\ \lambda
\right\rangle =\left\langle z,\ \lambda \right\rangle .
\end{equation*}
\end{definition}

For convenience and without loss of generality, we consider discrete
perturbations%
\begin{equation}
\begin{array}{rll}
N_{i}\dot{x}(t) & = & x(t)+f(t),\ t\geq 0 \\ 
x(0) & = & x_{0},%
\end{array}
\label{disPertur}
\end{equation}%
where $N_{i}$ is nonsingular, and%
\begin{equation}
\lim\nolimits_{i\rightarrow \infty }N_{i}=N.  \label{Niandxi0}
\end{equation}%
The solution of (\ref{disPertur}) is%
\begin{equation}
x_{i}(t)=\exp \{N_{i}{}^{-1}t\}x_{0}+\tint\nolimits_{0}^{t}\exp
\{N_{i}{}^{-1}(t-\tau )\}N_{i}{}^{-1}f(\tau )d\tau ,\text{ }t\geq 0.
\label{discretsolution}
\end{equation}%
Then we need to explore, in the sense of Definition \ref{DefOne}, the
convergence of the solution sequence $\{x_{i}\}_{i=1}^{\infty }$ to the
solution (\ref{Distributionsolution}).

In following, except $\delta ^{(k)},$ the $k$-th order derivative notation $%
z^{(k)}$ will always be in the ordinary sense according to pointwise
differentiation. In the case $z\in \mathcal{C}^{k}({\mathbb{R}}_{+},{\mathbb{%
R}}^{n}),$ notation $z^{(k)}(0)$ is understood as that from right hand. We
always assume $f\in \mathcal{C}^{q-1}({\mathbb{R}}_{+},{\mathbb{R}}^{n})$ in
this paper, where $q$ is the nilpotency index of $N,$ to guarantee the
distributional solution having the expression (\ref{Distributionsolution}).

\section{Uniqueness}

For a perturbation manner given by $\{N_{i}\}_{i=1}^{\infty },$ the solution
sequence $\{x_{i}(t)\}_{i=1}^{\infty }$ may not converge. But we will prove
that if it does, then the limit must be the solution (\ref%
{Distributionsolution}) of the reduced system (\ref{eqdiffer}), not
dependent of the perturbation manner. This generalizes Theorem 2 in \cite%
{Cobb1982}.

\begin{lemma}
\cite[p.\thinspace 21]{Gelfand1965}\label{Lemma1}Let $z\in \mathcal{C}^{k}({%
\mathbb{R}}_{+},{\mathbb{R}}^{n}).$ Then we have%
\begin{equation}
\mathcal{D}_{\mathrm{d}}^{(k)}z=z^{(k)}+\tsum\nolimits_{j=0}^{k-1}\delta
^{(j)}z^{(k-1-j)}(0).  \label{daoshuguanxi}
\end{equation}
\end{lemma}

Note that, according to the convention in Section \ref{Section2}, the
precise meaning of (\ref{daoshuguanxi}) is%
\begin{equation*}
\mathcal{D}_{\mathrm{d}}^{(k)}\mathcal{E(}z)=\mathcal{E(}z^{(k)})+\tsum%
\nolimits_{j=0}^{k-1}\delta ^{(j)}z^{(k-1-j)}(0).
\end{equation*}

\begin{lemma}
\label{Lemmadaoshu}$x_{i}(t)\in \mathcal{C}^{q}({\mathbb{R}}_{+},{\mathbb{R}}%
^{n})$ and for $m=1,2,\ldots ,q,$%
\begin{equation}
x_{i}^{(m)}(t)=N_{i}^{-m}x_{i}(t)+\tsum%
\nolimits_{l=1}^{m}N_{i}^{-l}f^{(m-l)}(t),\text{ }t\geq 0.
\label{casefor(m)}
\end{equation}
\end{lemma}

\begin{proof}
Firstly, we prove the case $m=1.$%
\begin{eqnarray}
x_{i}^{(1)}(t) &=&N_{i}^{-l}e^{N_{i}^{-1}t}x_{0}+\left(
e^{N_{i}^{-1}t}\tint\nolimits_{0}^{t}e^{-N_{i}^{-1}\tau }N_{i}^{-1}f(\tau
)d\tau \right) ^{\prime }  \notag \\
&=&N_{i}^{-l}\left(
e^{N_{i}^{-1}t}x_{0}+e^{N_{i}^{-1}t}\tint\nolimits_{0}^{t}e^{-N_{i}^{-1}\tau
}N_{i}^{-1}f(\tau )d\tau \right)
+e^{N_{i}^{-1}t}e^{-N_{i}^{-1}t}N_{i}^{-1}f(t)  \notag \\
&=&N_{i}^{-l}x_{i}(t)+N_{i}^{-1}f(t).  \label{Caseforone}
\end{eqnarray}

Secondly, supposing that the case $m$ holds, we prove the case $m+1.$
Differentiating two sides of (\ref{casefor(m)}) gives%
\begin{equation}
x_{i}^{(m+1)}(t)=N_{i}^{-m}x_{i}^{(1)}(t)+\tsum%
\nolimits_{l=1}^{m}N_{i}^{-l}f^{(m+1-l)}(t).  \label{casekplus1}
\end{equation}%
Substituting (\ref{Caseforone}) in (\ref{casekplus1}) gives the result
immediately.
\end{proof}

Combining Lemmas 1 and 2, we have

\begin{lemma}
\label{kjiedaoshu}For $k=1,2,\ldots ,q,$ we have%
\begin{equation}
\mathcal{D}_{\mathrm{d}}^{(k)}x_{i}=N_{i}^{-k}x_{i}+\tsum%
\nolimits_{l=1}^{k}N_{i}^{-l}f^{(k-l)}+\tsum\nolimits_{j=0}^{k-1}\delta
^{(j)}\left(
N_{i}^{-(k-1-j)}x_{0}+\tsum\nolimits_{l=1}^{k-1-j}N_{i}^{-l}f^{(k-1-j-l)}(0)%
\right) .  \label{kjiefenbudaoshu}
\end{equation}
\end{lemma}

\begin{lemma}
\cite[p.28]{Gelfand1965}\label{daoshushoulian}Let $\{z_{i}\}_{i=1}^{\infty }$
$\subset $ $\mathcal{C}_{C}^{\infty }(\mathbb{R},\mathbb{R}^{n})^{\prime }$
and $z$ $\in $ $\mathcal{C}_{C}^{\infty }(\mathbb{R},\mathbb{R}^{n})^{\prime
}.$ If $\lim_{i\rightarrow \infty }z_{i}$ $=$ $z,$ then for every $k\geq 1,$ 
\begin{equation*}
\lim_{i\rightarrow \infty }\mathcal{D}_{\mathrm{d}}^{(k)}z_{i}=\mathcal{D}_{%
\mathrm{d}}^{(k)}z.
\end{equation*}
\end{lemma}

\begin{theorem}
If $\{x_{i}\}_{i=1}^{\infty }$ converges, then $\lim_{i\rightarrow \infty
}x_{i}=x$.
\end{theorem}

\begin{proof}
Let $k=q,$ the index of nilpotency of $N,$ in (\ref{kjiefenbudaoshu}).
Multiplying two sides from left by $N_{i}^{q}$ gives%
\begin{eqnarray}
N_{i}^{q}\mathcal{D}_{\mathrm{d}}^{(q)}x_{i}
&=&x_{i}+\tsum\nolimits_{l=1}^{q}N_{i}^{q-l}f^{(q-l)}+\tsum%
\nolimits_{j=0}^{q-1}\delta ^{(j)}\left(
N_{i}^{q-(q-1-j)}x_{0}+\tsum%
\nolimits_{l=1}^{q-1-j}N_{i}^{q-l}f^{(q-1-j-l)}(0)\right)  \notag \\
&=&x_{i}+\tsum\nolimits_{l=0}^{q-1}N_{i}^{l}f^{(l)}+\tsum%
\nolimits_{j=0}^{q-1}\delta ^{(j)}N_{i}^{j+1}\left(
x_{0}+\tsum\nolimits_{m=0}^{q-2-j}N_{i}^{m}f^{(m)}(0)\right) .
\label{jiedekejiedaoshu}
\end{eqnarray}%
Letting $i\rightarrow \infty $ and noting that $N_{i}\rightarrow N,$ we
obtain%
\begin{equation*}
N^{q}\lim_{i\rightarrow \infty }\mathcal{D}_{\mathrm{d}}^{(q)}x_{i}=\lim_{i%
\rightarrow \infty
}x_{i}+\tsum\nolimits_{l=0}^{q-1}N^{l}f^{(l)}+\tsum\nolimits_{j=0}^{q-1}%
\delta ^{(j)}N^{j+1}\left(
x_{0}+\tsum\nolimits_{m=0}^{q-2-j}N^{m}f^{(m)}(0)\right)
\end{equation*}%
from Lemma \ref{daoshushoulian}. Noting that $N^{q}=0$ and $%
N^{j+1}\sum_{m=q-2-j+1}^{q-1}N^{m}=0,$ we have%
\begin{eqnarray*}
\lim_{i\rightarrow \infty }x_{i}
&=&-\tsum\nolimits_{l=0}^{q-1}N^{l}f^{(l)}-\tsum\nolimits_{j=0}^{q-2}\delta
^{(j)}N^{j+1}\left( x_{0}+\tsum\nolimits_{m=0}^{q-1}N^{m}f^{(m)}(0)\right) \\
&=&-\tsum\nolimits_{l=0}^{q-1}N^{l}f^{(l)}-\tsum\nolimits_{k=1}^{q-1}\delta
^{(k-1)}N^{k}\left( x_{0}+\tsum\nolimits_{m=0}^{q-1}N^{m}f^{(m)}(0)\right) .
\end{eqnarray*}%
This completes the proof.
\end{proof}

\section{Convergence}

In this section, we will give a condition on perturbation to guarantee
convergence. An example satisfying the condition shows the existence of
convergent perturbation. This gives a generalization to Theorem 1 in \cite%
{Cobb1982}.

\begin{lemma}
\label{LemmaA}If the number sequence $\{\tint\nolimits_{0}^{+\infty
}||N_{i}^{k}e^{N_{i}^{-1}t}||dt,$ $i=1,2,\ldots \}$ is bounded for some $%
k\geq 0,$ and $f\in \mathcal{C}^{q+k}({\mathbb{R}}_{+},{\mathbb{R}}^{n})\cap 
\mathcal{L}^{1}({\mathbb{R}}_{+},{\mathbb{R}}^{n})$ then $%
\{x_{i}\}_{i=1}^{\infty }$ converges.
\end{lemma}

\begin{proof}
Under the boundedness assumption, the sequence $%
\{N_{i}^{q+1+k}e^{N_{i}^{-1}t}x_{0}\}_{i=1}^{\infty }$ converges to $0$ in
the sense of Definition \ref{DefOne} by Lemma 1 in \cite{Cobb1982}. Let $%
h\in \mathcal{C}_{C}^{\infty }(\mathbb{R},\mathbb{R}^{n})$ with $%
||h(t)||\leq C$ for $\forall t\in \mathbb{R}$. Since%
\begin{eqnarray*}
&&\left\vert \left\langle
N_{i}^{q+1+k}\tint\nolimits_{0}^{t}e^{N_{i}^{-1}(t-\tau )}N_{i}^{-1}f(\tau
)d\tau ,\ h\right\rangle \right\vert \\
&\leq &\tint\nolimits_{0}^{+\infty }||h(t)||\cdot \left(
\tint\nolimits_{0}^{t}||N_{i}^{q+k}e^{N_{i}^{-1}(t-\tau )}||\cdot ||f(\tau
)||d\tau \right) dt \\
&\leq &\tint\nolimits_{0}^{+\infty }||f(\tau )||\cdot \left(
\tint\nolimits_{\tau }^{+\infty }||N_{i}^{q+k}e^{N_{i}^{-1}(t-\tau )}||\cdot
||h(t)||dt\right) d\tau \\
&\leq &\tint\nolimits_{0}^{+\infty }||f(\tau )||d\tau \cdot
||N_{i}^{q}||C\tint\nolimits_{0}^{+\infty }||N_{i}^{k}e^{N_{i}^{-1}t}||dt \\
&\rightarrow &0
\end{eqnarray*}%
by the assumptions (note that $||N_{i}^{q}||\rightarrow ||N^{q}||=0$), the
sequence $\{N_{i}^{q+1+k}\tint\nolimits_{0}^{t}e^{N_{i}^{-1}(t-\tau
)}N_{i}^{-1}f(\tau )d\tau $ $:$ $i=1,2,$ $\ldots \}$ converges to $0$ in $%
\mathcal{C}_{C}^{\infty }(\mathbb{R},\mathbb{R}^{n})^{\prime }$ also. So we
have 
\begin{equation*}
N_{i}^{q+1+k}x_{i}(t)=N_{i}^{q+1+k}e^{N_{i}^{-1}t}x_{0}+N_{i}^{q+1+k}\tint%
\nolimits_{0}^{t}e^{N_{i}^{-1}(t-\tau )}N_{i}^{-1}f(\tau )d\tau ,\text{ }%
t\geq 0
\end{equation*}%
converges to $0$ in $\mathcal{C}_{C}^{\infty }(\mathbb{R},\mathbb{R}%
^{n})^{\prime }.\ $By Lemma \ref{daoshushoulian}, we have 
\begin{equation}
\lim_{i\rightarrow \infty }\mathcal{D}_{\mathrm{d}%
}^{(q+1+k)}(N_{i}^{q+1+k}x_{i})=\lim_{i\rightarrow \infty }N_{i}^{q+1+k}%
\mathcal{D}_{\mathrm{d}}^{(q+1+k)}x_{i}=0.  \label{proofeqone}
\end{equation}%
On the other hand, since $f$ $\in $ $\mathcal{C}^{q+k}({\mathbb{R}}_{+},{%
\mathbb{R}}^{n})$, we have%
\begin{eqnarray}
N_{i}^{q+1+k}\mathcal{D}_{\mathrm{d}}^{(q+1+k)}x_{i}
&=&x_{i}+\tsum\nolimits_{l=0}^{(q+1+k)-1}N_{i}^{l}f^{(l)}  \label{proofeqtwo}
\\
&&+\tsum\nolimits_{j=0}^{(q+1+k)-1}\delta ^{(j)}N_{i}^{j+1}\left(
x_{0}+\tsum\nolimits_{m=0}^{(q+1+k)-2-j}N_{i}^{m}f^{(m)}(0)\right)  \notag
\end{eqnarray}%
like (\ref{jiedekejiedaoshu}). From (\ref{proofeqone}) and (\ref{proofeqtwo}%
) we see the existence of $\lim_{i\rightarrow \infty }x_{i}$ and 
\begin{equation*}
\lim_{i\rightarrow \infty
}x_{i}=-\tsum\nolimits_{l=0}^{(q+1+k)-1}N^{l}f^{(l)}-\tsum%
\nolimits_{j=0}^{(q+1+k)-1}\delta ^{(j)}N^{j+1}\left(
x_{0}+\tsum\nolimits_{m=0}^{(q+1+k)-2-j}N^{m}f^{(m)}(0)\right) .
\end{equation*}%
Noting that $N^{q}=0,$ we see that it equals $x$ by (\ref%
{Distributionsolution}).
\end{proof}

We intend to weaken the higher differentiability requirement for $f\in 
\mathcal{C}^{q+k}({\mathbb{R}}_{+},{\mathbb{R}}^{n})$ in Lemma \ref{LemmaA}.
Again, we note that $f$ is always assumed in $\mathcal{C}^{q-1}({\mathbb{R}}%
_{+},{\mathbb{R}}^{n}).$

\begin{lemma}
\label{LemmaA'}Suppose $f\in \mathcal{L}^{1}({\mathbb{R}}_{+},{\mathbb{R}}%
^{n}).$ If the number sequence $\{\tint\nolimits_{0}^{+\infty
}||N_{i}^{k}e^{N_{i}^{-1}t}||dt,$ $i=1,2,\ldots \}$ is bounded for some $%
k\geq 0,$ then $\{x_{i}\}_{i=1}^{\infty }$ converges.
\end{lemma}

\begin{proof}
We only prove the result in\ the case $k=0.$ That for $k\geq 1$ can be
proved by some slight modification. Note that $f\in \mathcal{C}^{q-1}({%
\mathbb{R}}_{+},{\mathbb{R}}^{n})$ but maybe $f\notin $ $\mathcal{C}^{q+0}({%
\mathbb{R}}_{+},{\mathbb{R}}^{n})$ $=\mathcal{C}^{q}({\mathbb{R}}_{+},{%
\mathbb{R}}^{n}).$

Differentiating two sides of (\ref{jiedekejiedaoshu}) gives%
\begin{equation}
N_{i}^{q}\mathcal{D}_{\mathrm{d}}^{(q+1)}x_{i}=\mathcal{D}_{\mathrm{d}%
}x_{i}+\tsum\nolimits_{l=0}^{q-1}N_{i}^{l}\mathcal{D}_{\mathrm{d}%
}f^{(l)}+\tsum\nolimits_{j=0}^{q-1}\delta ^{(j+1)}N_{i}^{j+1}\left(
x_{0}+\tsum\nolimits_{m=0}^{q-2-j}N_{i}^{m}f^{(m)}(0)\right) .
\label{qjia1jiedaoshu}
\end{equation}%
Noting that $x_{i}\in \mathcal{C}^{1}({\mathbb{R}}_{+},{\mathbb{R}}^{n})$
and $f^{(l)}\in \mathcal{C}^{1}({\mathbb{R}}_{+},{\mathbb{R}}^{n})$ for $%
l=0,1,\ldots ,q-2,$ it follows from Lemma \ref{Lemma1} that 
\begin{eqnarray*}
\mathcal{D}_{\mathrm{d}}x_{i} &=&\dot{x}_{i}+\delta \cdot x_{0} \\
&=&N_{i}^{-1}x_{i}+N_{i}^{-1}f+\delta \cdot x_{0},
\end{eqnarray*}%
and%
\begin{equation*}
\mathcal{D}_{\mathrm{d}}f^{(l)}=f^{(l+1)}+\delta \cdot f^{(l)}(0)
\end{equation*}%
for $l=0,1,\ldots ,q-2.$ Substituting in (\ref{qjia1jiedaoshu}) gives 
\begin{eqnarray*}
N_{i}^{q}\mathcal{D}_{\mathrm{d}}^{(q+1)}x_{i}
&=&N_{i}^{-1}x_{i}+N_{i}^{-1}f+\delta \cdot x_{0} \\
&&+\tsum\nolimits_{l=0}^{q-2}N_{i}^{l}f^{(l+1)}+\delta \cdot
\tsum\nolimits_{l=0}^{q-2}N_{i}^{l}f^{(l)}(0)+N_{i}^{q-1}\mathcal{D}_{%
\mathrm{d}}f^{(q-1)} \\
&&+\tsum\nolimits_{j=0}^{q-1}\delta ^{(j+1)}N_{i}^{j+1}\left(
x_{0}+\tsum\nolimits_{m=0}^{q-2-j}N_{i}^{m}f^{(m)}(0)\right) \\
&=&N_{i}^{-1}x_{i}+N_{i}^{-1}\tsum%
\nolimits_{l=-1}^{q-2}N_{i}^{l+1}f^{(l+1)}+N_{i}^{q-1}\mathcal{D}_{\mathrm{d}%
}f^{(q-1)} \\
&&+N_{i}^{-1}\tsum\nolimits_{j=-1}^{q-1}\delta ^{(j+1)}N_{i}^{j+1+1}\left(
x_{0}+\tsum\nolimits_{m=0}^{q-2-j}N_{i}^{m}f^{(m)}(0)\right) -\delta \cdot
N_{i}^{q-1}f^{(q-1)}(0).
\end{eqnarray*}%
Then we have%
\begin{eqnarray*}
N_{i}^{q+1}\mathcal{D}_{\mathrm{d}}^{(q+1)}x_{i}
&=&x_{i}+\tsum\nolimits_{l=-1}^{q-2}N_{i}^{l+1}f^{(l+1)}+N_{i}^{q}\mathcal{D}%
_{\mathrm{d}}f^{(q-1)} \\
&&+\tsum\nolimits_{j=-1}^{q-1}\delta ^{(j+1)}N_{i}^{j+1+1}\left(
x_{0}+\tsum\nolimits_{m=0}^{q-2-j}N_{i}^{m}f^{(m)}(0)\right)
-N_{i}^{q}f^{(q-1)}(0). \\
&&.
\end{eqnarray*}%
Noting that 
\begin{equation*}
\lim_{i\rightarrow \infty }N_{i}^{q}\mathcal{D}_{\mathrm{d}}f^{(q-1)}=N^{q}%
\mathcal{D}_{\mathrm{d}}f^{(q-1)}=0,
\end{equation*}%
the remainder thing is similar to the proof of Lemma \ref{LemmaA}.
\end{proof}

We need to weaken the integrability requirement $f\in \mathcal{L}^{1}({%
\mathbb{R}}_{+},{\mathbb{R}}^{n}).$

\begin{lemma}
\label{LemmaC}For any $b>0,$ there exists $f_{b}\in \mathcal{C}^{q-1}({%
\mathbb{R}}_{+},{\mathbb{R}}^{n})$ $\cap \mathcal{L}^{1}({\mathbb{R}}_{+},{%
\mathbb{R}}^{n})$ such that%
\begin{equation}
f_{b}(t)=f(t),\ \forall t\leq b.  \label{fb}
\end{equation}
\end{lemma}

\begin{proof}
One can construct a (unique) polynomial $P(t)$\ of degree $(2q-1)$ such that%
\begin{equation*}
P^{(k)}(b)=f^{(k)}(b),\ P^{(k)}(b+1)=0
\end{equation*}%
for $k=0,1,\ldots ,q-1$ (see \cite[p. 88]{HeXuChu1980}). Then we define%
\begin{equation*}
f_{b}(t)=\left\{ 
\begin{array}{ll}
f(t), & \text{if }0\leq t\leq b, \\ 
P(t), & \text{if }b<t\leq b+1, \\ 
0, & \text{if }t>b+1,%
\end{array}%
\right.
\end{equation*}%
which satisfies the requirement.
\end{proof}

\begin{theorem}
\label{ThTwo}If the number sequence $\{\tint\nolimits_{0}^{+\infty
}||N_{i}^{k}e^{N_{i}^{-1}t}||dt,$ $i=1,2,\ldots \}$ is bounded for some $%
k\geq 0,$ then $\{x_{i}\}_{i=1}^{\infty }$ converges.
\end{theorem}

\begin{proof}
Arbitrarily choose $h\in \mathcal{C}_{C}^{\infty }(\mathbb{R},\mathbb{R}%
^{n}).$ Then we have%
\begin{equation*}
h(t)=0,\ \forall t\geq b
\end{equation*}%
for some $b>0.$ Let $f_{b}\in \mathcal{C}^{q-1}({\mathbb{R}}_{+},{\mathbb{R}}%
^{n})\cap \mathcal{L}^{1}({\mathbb{R}}_{+},{\mathbb{R}}^{n})$ with (\ref{fb}%
). Then by Lemma \ref{LemmaA'}, the sequence%
\begin{equation*}
y_{i}(t)=e^{N_{i}^{-1}t}x_{0}+\tint\nolimits_{0}^{t}e^{N_{i}^{-1}(t-\tau
)}N_{i}^{-1}f_{b}(\tau )d\tau ,\text{ }t\geq 0
\end{equation*}%
converges to%
\begin{equation*}
y(t)=-\tsum\nolimits_{i=0}^{q-1}N^{i}f_{b}^{(i)}(t)-\tsum%
\nolimits_{k=1}^{q-1}\delta ^{(k-1)}(t)N^{k}\left\{
x_{0}+\tsum\nolimits_{i=0}^{q-1}N^{i}f_{b}^{(i)}(0)\right\} ,\text{ }t\geq 0
\end{equation*}%
in the sense of Definition \ref{DefOne}.\ By direct computation we can get%
\begin{equation*}
\left\langle x_{i},\ h\right\rangle =\left\langle y_{i},\ h\right\rangle ,%
\text{ }i=1,2,\ldots
\end{equation*}%
and%
\begin{equation*}
\left\langle x,\ h\right\rangle =\left\langle y,\ h\right\rangle .
\end{equation*}%
Therefore $\lim_{i\rightarrow \infty }\left\langle x_{i},\ h\right\rangle
=\left\langle x,\ h\right\rangle ,$ and this completes the proof.
\end{proof}

\begin{example}
\label{LemmaB}Set $N_{i}=N-\frac{1}{i}I,$ $i=1,2,\ldots $. Then $%
\{\tint\nolimits_{0}^{+\infty }||N_{i}^{k}e^{N_{i}^{-1}t}||dt,$ $%
i=1,2,\ldots \}$ is bounded for some $k\geq 0$ (see Lemma 2 in \cite%
{Cobb1982}). So according to this perturbation manner, Theorem \ref{ThTwo}
guarantees that the solution sequence $\{x_{i}\}_{i=1}^{\infty }$ of the
perturbed systems (\ref{disPertur}) converges to the solution $x$ of the
singular system (\ref{eqdiffer}).
\end{example}

\section{Conclusions}

As an idealized model, the nonhomogeneous singular system can approximate
some singularly perturbed systems well in a sense of distribution theory. A
future work is to give some condition easy to verify on perturbations to
guarantee convergence.


\begin{thebibliography}{99}
\bibitem{Verghese1981} G.~C.~Verghese, B.~C.~L\'{e}vy, T.~Kailath,
\textquotedblleft A generalized state-space for singular
systems,\textquotedblright\ \textit{IEEE Transactions on Automatic Control},
Vol. AC-26, No. 4, pp. 811--831, 1981.

\bibitem{Dai1989} L. Dai, \textit{Singular Control Systems}, Springer-Verlag
Berlin, Heidelberg, 1989.

\bibitem{Kokotovic1976} P. V. Kokotovic, R. E. O'Malley, Jr., and P.
Sannuti, \textquotedblleft Singular perturbation and order reduction in
control theory---an overview,\textquotedblright\ \textit{Automatica}, Vol.
12, pp. 123--132, 1976.

\bibitem{Kokotovic1980} P. V. Kokotovic, J. J. Allemong, J. R. Winkelman,
and J. H. Chow, \textquotedblleft Singular perturbation and iterative
separation of time scales,\textquotedblright\ \textit{Automatica}, Vol. 16,
pp. 23--33, 1980.

\bibitem{Francis1979} B. A. Francis, \textquotedblleft The optimal
linear-quadratic time-invariant regulator with cheap
control,\textquotedblright\ \textit{IEEE Transactions on Automatic Control},
Vol. AC-24, pp. 616--621, 1979.

\bibitem{Francis1982} B. A. Francis, \textquotedblleft Convergence in the
boundary layer for singularly perturbed equations,\textquotedblright\ 
\textit{Automatica}, Vol. 18, No. 2, pp. 57--62, 1982.

\bibitem{Cobb1982} D. Cobb, \textquotedblleft On the solutions of linear
differential equations with singular coefficients,\textquotedblright\ 
\textit{Journal of Differential Equations}, Vol. 46, pp. 310--323, 1982.

\bibitem{Cobb2006} D. Cobb and J. Eapen,\ \textquotedblleft High-gain state
feedback analysis based on singular system theory,\textquotedblright\ 
\textit{SIAM J. Control Optim.} Vol. 44, No. 6, pp. 2210--2232, 2006.

\bibitem{Hinrichsen1997} D. Hinrichsen and J. O'Halloran, \textquotedblleft
Limits of generalized state space systems under proportional and derivative
feedback,\textquotedblright\ \textit{Math. Control Signals Systems}, Vol.
10, pp. 97--124, 1997.

\bibitem{Gantmacher} F. R. Gantmacher, \textit{The Theory of Matrices, }Vol.
2, Higher Education Press, Beijing, 1955 (Chinese translation from Russian).

\bibitem{Muller2005} P. C. M\"{u}ller, \textquotedblleft Remark on the
solution of linear time-invariant descriptor systems,\textquotedblright\ 
\textit{PAMM Proc. Appl. Math. Mech.}, 5, pp. 175--176, 2005.

\bibitem{Yan2005} Z. Yan and G. Duan, \textquotedblleft Time domain solution
to descriptor variable systems,\textquotedblright\ \textit{IEEE Transactions
on Automatic Control}, Vol. 50, No. 11, pp. 1796--1799, 2005.

\bibitem{Yan2007} Z. Yan, \textquotedblleft Geometric analysis of impulse
controllability for descriptor system,\textquotedblright\ \textit{Systems
and Control Letters}, Vol. 56, No. 1, pp. 1--6, 2007.

\bibitem{O'Malley1979} R. E. O'Malley, Jr., \textquotedblleft A singular
singularly-perturbed linear boundary value problem,\textquotedblright\ 
\textit{SIAM J. Math. Anal.} Vol. 10, No. 4, pp. 695--708, 1979.

\bibitem{Campbell1980} S. L. Campbell, \textit{Singular Systems of
Differential Equations}, Pitman, London, 1980.

\bibitem{Gelfand1965} I. M. Gelfand and G. E. Shilov, \textit{Generalized
Functions}, Vol. I, Science Press, Beijing, 1965 (Chinese translation from
Russian).

\bibitem{HeXuChu1980} X. He, Y. Su, and X. Bao, \textit{A Short Course in
Numerical Mathematics}, Higher Education Press, Beijing, 1980 (in Chinese).
\end{thebibliography}
\end{document}